\documentclass[final,12pt]{elsarticle}
\usepackage{xcolor,soul,cancel}
\usepackage[color]{showkeys}
\usepackage{amssymb,amsmath,amsthm,mathtools}
\usepackage{graphicx,float,enumitem,hyperref}
\newlist{alphalist}{enumerate}{1} 
\setlist[alphalist]{label=(\alph*)} 
\journal{arXiv}
\definecolor{refkey}{rgb}{0,1,1}
\definecolor{labelkey}{rgb}{1,0,0}

\numberwithin{equation}{section}

\usepackage[a4paper,margin=27mm]{geometry}
\usepackage{amsmath,amssymb,amsthm,mathtools}
\usepackage{microtype}
\usepackage{enumitem}

\newtheorem{thm}{Theorem}[section]
\newtheorem{lem}[thm]{Lemma}
\newtheorem{cor}[thm]{Corollary}
\newtheorem{prop}[thm]{Proposition}
\newtheorem{ex}[thm]{Example}
\newtheorem{defi}[thm]{Definition}

\theoremstyle{remark}

\newcommand{\R}{\mathbb R}
\newcommand{\E}{\mathbb E}
\newcommand{\Var}{\operatorname{Var}}
\newcommand{\Cov}{\operatorname{Cov}}
\newcommand{\Unif}{\operatorname{Unif}}

\newcommand{\Gini}{\mathcal D}

\theoremstyle{definition}

\newtheorem{rem}[thm]{Remark}

\begin{document}
\begin{frontmatter}
\title{The Zaporozhets–Tarasov Conjecture on Mean Distances}
\author[eric]{Eric Shen}
\ead{erick.2013@yandex.ru}
\address[eric]{Moscow State University, Moscow, 119991, Russia \\
Institute for Numerical Mathematics, Russian Academy of Sciences, Russia}
\date{August 2026}

\begin{abstract}
For a convex body  $K\subset\R^d$  let $\Delta(K)$ be the expected distance between two independent uniform points of $K$, and let $\theta(K)$ be the expected distance between two independent uniform points on $\partial K$. The Zaporozhets–Tarasov conjecture asserts $\Delta(K) \le \theta(K)$. We prove this conjecture in case $d=2$. In addition, we give a six-vertex convex polytope in $\mathbb R^3$ for which the reverse strict inequality holds, and obtain counterexamples in every dimension $d \ge 3$ by taking products with segments. This settles the conjecture in every dimension $d \geq 2$.
\end{abstract}

\end{frontmatter}

\section{Introduction}
Let $K \subset \mathbb{R}^d$ be a compact convex body. The conjecture of Zaporozhets and Tarasov asks whether
$$
\Delta(K):=\E|X_1-X_2|\leq \E|B_1-B_2|=: \theta(K)
$$
for every convex body $K\subset\R^d$, where $X_1,X_2$ are independent  uniform in $K$ with respect to the normalized Lebesgue measure and $B_1,B_2$ are independent uniform on $\partial K$ with respect to the normalized \((d-1)\)-dimensional Hausdorff measure. The conjecture was proposed by D. N. Zaporozhets and A. S. Tarasov in 2019 through informal communication. To the best of our knowledge, its first explicit formulation in the literature is due to Tokmachev \cite{TokmachevBoundary} (Eq. (3) on p. 3), who attributes it to them; see also discussion before Conjecture 1.1 in the work~\cite{Lotnikov}.

Lotnikov proved the planar centrally symmetric case and an analogue for sufficiently high moments in arbitrary dimension \cite{Lotnikov}. Tokmachev proved the conjecture, in fact a convex-order statement, for planar bodies whose area and boundary centroids coincide and for a class of circumscribed polytopes \cite{Tokmachev2026}. Related extremal questions for $\Delta$ and $\theta$ were studied in \cite{BonnetEtAl,TokmachevBoundary}.
In particular, Bonnet, Gusakova, Th\"ale and Zaporozhets proved the sharp inequalities
$$
    \frac{7}{60}
    <
    \frac{\Delta(K)}{\operatorname{per} K}
    <
    \frac{1}{6}
$$
for planar convex bodies~\cite{BonnetEtAl}, where $\operatorname{per} K$ stands for perimeter of $K$.  It is also shown that the upper bound is approached by thin rectangles converging to a doubled segment.
According to Tokmachev~\cite[Eq.~(5)]{Tokmachev2026},
Gusakova and Zaporozhets suggested an analogous pair of inequalities for $\theta$
$$
    \frac{1}{6}
    <
    \frac{\theta(K)}{\operatorname{per} K}
    \leq
    \frac{2}{\pi^2}.$$
The upper bound, with equality only for the disk, was proved by
Tokmachev~\cite{TokmachevBoundary}; the complementary lower bound was established by the author after the initial submission of the present paper, see Section 5 in \cite{ShenAug2026}. Combined with the upper bound for $\Delta(K)$ from \cite{BonnetEtAl}, it yields another proof of the planar Zaporozhets–Tarasov conjecture. 
\par We prove the Zaporozhets-Tarasov conjecture in the case $d=2$, see Corollary~\ref{corol:conjecture_res}. The result is mostly self-contained, in particular, it does not use the aforementioned inequalities on $\Delta(K)$ from \cite{BonnetEtAl} and is obtained as a corollary of a more general statement, see Theorem~\ref{thm:endpoint}. We also prove an analogous inequality for the second moment, see Corollary~\ref{cor:sec_mom_planar}. In addition to that, we show that the conjecture is false in every dimension $d\geq3$, see Example~\ref{ex:counter-ex}, Figure~\ref{fig:body}, and Proposition~\ref{prop:counter-ex}.

\begin{defi}
    For a probability measure $\rho$ on $\R$, write
$$
 \Gini(\rho)=\iint_{\R^2}|x-y|\,d\rho(x)d\rho(y).
$$
This is the \textit{Gini mean difference} of $\rho$.
\end{defi}

We slightly abuse this notation throughout the text, writing $\Gini(X)$ with $X$ being a random variable or a probability density. In either case, $\Gini(X)$ is understood as $\Gini(\rho_X)$, where $\rho_X$ denotes the corresponding probability measure.

We use the following standard identities for the Gini mean
difference; see also \cite[Sec.~2.1.1]{Yitzhaki}.

\begin{lem}\label{lem:gini-identities}
Let $\rho,\sigma$ be probability measures on $\R$ with finite
first moments, and let $F,G$ be their distribution functions.
Then
\begin{equation}\label{eq:gini-cdf}
\Gini(\rho)=2\int_\R F(x)(1-F(x))\,dx.
\end{equation}
Moreover, for $0\le\lambda\le1$,
\begin{equation}\label{eq:gini-mixture}
\Gini(\lambda\rho+(1-\lambda)\sigma)
=\lambda\Gini(\rho)+(1-\lambda)\Gini(\sigma)+2\lambda(1-\lambda)\int_\R(F-G)^2\,dx.
\end{equation}
Finally, if $X$ has an atomless distribution $\rho$, then
\begin{equation}\label{eq:gini-cov}
\Gini(X)=4\Cov(X,F(X)).
\end{equation}
\end{lem}

\begin{proof}
Let $X,X'$ be independent with law $\rho$. Since
\[
|X-X'|=\int_\R
\left|\mathbf{1}_{\{X\le x\}}-
\mathbf{1}_{\{X'\le x\}}\right|\,dx,
\]
taking expectations and applying Tonelli's theorem gives
\eqref{eq:gini-cdf}.

The distribution function of $\lambda\rho+(1-\lambda)\sigma$
is $H=\lambda F+(1-\lambda)G$. The identity
\[
H(1-H)=\lambda F(1-F)+(1-\lambda)G(1-G)
+\lambda(1-\lambda)(F-G)^2
\]
then proves \eqref{eq:gini-mixture}.

Finally, if $\rho$ is atomless, then $F(X)$ is uniform on
$[0,1]$. By independence and symmetry,
\[
2\Cov(X,F(X))=2\E[X\mathbf{1}_{\{X'\le X\}}]-\E X=\E[(X-X')\mathbf{1}_{\{X'\le X\}}]=\frac12\E|X-X'|.
\]
This proves \eqref{eq:gini-cov}.
\end{proof}

The main result immediately follows from a stronger directional inequality.

\begin{thm}\label{thm:main}
Let $K\subset\mathbb R^2$ be a convex body. Fix a unit vector $u$. Let $p_u$ be the pushforward of normalized area measure on $K$ under $x\mapsto\langle x,u\rangle$, and let $q_u$ be the corresponding pushforward of normalized arclength measure on $\partial K$. Then
\begin{equation}\label{eq:directional}
 \Gini(q_u)\ge \Gini(p_u).
\end{equation}
The inequality is strict if $u$ is not perpendicular to any nondegenerate segment contained in $\partial K$.
\end{thm}

\medskip
\noindent\textbf{Plan.}
After rotating and rescaling, the projection of $K$ is $[0,1]$ and
$$
 K=\{(x,y):0\le x\le1,\ g(x)\le y\le f(x)\},
$$
where $f$ is concave and $g$ is convex. Put
$$
 W=f-g,\qquad a=f',\qquad b=g'.
$$
The projected area density is proportional to $W$, while the projected boundary density is proportional to
$$
 \sqrt{1+a^2}+\sqrt{1+b^2}.
$$
The proof has two steps. First, we find two densities $\varphi$ and $\psi$, nonincreasing and nondecreasing, respectively, such that the projected arclength density $q$ satisfies $q = \frac{\varphi+\psi}{2}$ while the projected area density $p$ is proportional to the difference of the CDFs of $\varphi$ and $\psi$. Then we apply a one-dimensional theorem saying that an endpoint mixture is more dispersed than the normalized occupation measure between the endpoints.
The only substantial step is the second one, to which we now turn.

\section{The monotone-endpoint theorem}

\begin{thm}
    \label{thm:endpoint}
Let $\varphi$ be a nonincreasing probability density and $\psi$ a nondecreasing probability density on $[0,1]$. Let
$$
 \Phi(x)=\int_0^x\varphi(s)\,ds,
 \qquad
 \Psi(x)=\int_0^x\psi(s)\,ds.
$$
Put
$$
 H=\Phi-\Psi,
 \qquad
 A=\int_0^1H(x)\,dx,
$$
and assume $A>0$. Define
$$
 p(x)=\frac{H(x)}A,
 \qquad
 q(x)=\frac{\varphi(x)+\psi(x)}2.
$$
Then
$$
 \Gini(q)> \Gini(p).
$$
\end{thm}

 Since $\Phi$ is concave and $\Psi$ is convex, $\Phi(x)\ge x \ge \Psi(x)$.
 Thus $H\ge0$, and $p$ is indeed a probability density.

\begin{rem}
There is also a geometric reformulation.  Let $x(u)=\Phi^{-1}(u)$ and $y(u)=\Psi^{-1}(u)$. If $U\sim\Unif[0,1]$, then $ x(U)\sim \varphi,\,\,  y(U) \sim \psi$. Consider
$$
    L=\{(u,z):0\leq u\leq 1,\ x(u)\leq z\leq y(u)\}.
$$
Since $x$ is convex and $y$ is concave, $L$ is convex. Choosing one of the two endpoints
of the (uniformly distributed) vertical slice $[x(U),y(U)]$ with equal probability gives the law $q$.
Schematically,
$$
 q(z)dz=\frac12\int_0^1\bigl(\delta_{x(u)}+\delta_{y(u)}\bigr)(dz)\,du.
$$
 On the other hand, the length of the horizontal section of $L$ at height $z$ is
    $$\bigl|\{u:x(u)\leq z\leq y(u)\}\bigr|
    =\Phi(z)-\Psi(z)=H(z). $$
Consequently, $p$ is the second-coordinate distribution of a uniform
point of $L$, that is, choose a slice with a probability proportional to its length and then a random point in it. Explicitly, 
$$
 p(z)\,dz
 =\frac{\displaystyle\int_0^1
 \mathbf 1_{\{x(u)\le z\le y(u)\}}\,du}
 {\displaystyle\int_0^1(y(u)-x(u))\,du}\,dz.
$$ 
\end{rem}

Every nonincreasing probability density on $[0,1]$ is a mixture of uniform densities on intervals $[0,t]$, as shown by the following Proposition. This is a one-sided form of Khintchine’s representation theorem for unimodal distributions, see \cite{Khintchine}.

\begin{prop} \label{prop:monotone_factoriz}
Let $\varphi$ be a probability density on $[0,1]$ admitting a nonincreasing representative, which we also denote by $\varphi$. There is a random variable $T\in[0,1]$ such that, for $U\sim\Unif[0,1]$ independent of $T$, the variable
$$
 X=TU
$$
has density $\varphi$. Likewise, for every nondecreasing density $\psi$, there is $S\in[0,1]$ such that
$$
 Y=1-SV,
 \qquad V\sim\Unif[0,1] \,\, \text{independent of } S,
$$
has density $\psi$.
\end{prop}
    \begin{proof}
        
     As the law of $T$ one may take the Stieltjes mixing measure 
$$
 \rho(dt)=-t\,d\varphi(t)+\varphi(1-)\delta_1(dt),
$$
for which
$$
 \varphi(x)=\int_{[x,1]}\frac1t\,\rho(dt) \quad \text{for a.e. } x\in (0,1) .
$$ 
As $\varphi$ is nonincreasing, $\rho$ is nonnegative and the Stieltjes integration by parts gives $\rho ([0,1]) = \int_0^1 \varphi = 1.$
\end{proof}

We first record two consequences of Proposition~\ref{prop:monotone_factoriz}.

\begin{lem}\label{lem:low_bound_enpdpoint}
Let $X$ have a nonincreasing density supported on $[0,1]$, and let $Y$ have a nondecreasing one. Then
\begin{equation}\label{eq:Dmunu}
\begin{aligned}
 \Gini(X)&\ge 2\Var X+\frac23\E X(1-\E X),\\
 \Gini(Y)&\ge 2\Var Y+\frac23\E Y(1-\E Y).
\end{aligned}
\end{equation}
Moreover,
$$
 \E(X-\E X)^3\ge0,
 \qquad
 \E(Y-\E Y)^3\le0.
$$
\end{lem}
\begin{proof}
By Proposition~\ref{prop:monotone_factoriz}, there are mutually independent $T,S\in[0,1]$ and $U, V \sim \Unif[0,1]$ such that $X=TU$ and $Y=1-SV$. Since $U$ is uniform on $[0,1]$
and independent of $T$,
$$
 \E X=\frac{\E T}{2},
 \qquad
 \E X^2=\frac{\E T^2}{3},
 \qquad
 \E X^3=\frac{\E T^3}{4}.
$$
Consequently,
$$
 4\E(X-\E X)^3=\E T^3-2\E T\,\E T^2+(\E T)^3=\E\bigl[T(T-\E T)^2\bigr]\ge0.
$$
Applying the same argument to $1-Y=SV$ gives
$\E(Y-\E Y)^3\le0$.

Take $U,U'\sim\Unif[0,1]$ independent. Direct integration gives

$$ \E|tU-sU'|
 =\frac{t+s}{6}
 +\frac{(t-s)^2}{3\max(t,s)},$$

with the second term read as zero when $s=t=0$. Since $\max(s,t)\le1$,
$$
 \E|tU-sU'|
 \ge \frac{t+s}{6}+\frac{(t-s)^2}{3}.
$$
Now let $(T',U')$ be an independent copy of $(T,U)$. Then
$$\Gini(X) = \E|TU-T'U'| = \E_{T,T'}\E(|TU-T'U'| \, |\,T, T') \geq \E(\frac{T+T'}{6}+\frac{(T-T')^2}{3}).$$
Using that $T$ and $T'$ are independent gives
$$
 \Gini(X)\ge\frac{\E T}{3}+\frac23\Var T.
$$
Using
$$
 \Var T=3\Var X-(\E X)^2,
 \qquad \E T=2\E X,
$$
we obtain the first inequality in~\eqref{eq:Dmunu}. Applying the same argument to $1-Y$ gives the second one.
\end{proof}

\begin{lem}\label{lem:upper_bound}
Let $\varphi$, $\psi$, $p$ and $A$ be as in Theorem~\ref{thm:endpoint}, let $X, Y, Z$ have densities  $\varphi, \psi,$ and $p$, respectively, and put
$m=(\E X+\E Y)/2$. Then
$$
 \E(Z-m)^2
 \le\frac{\Var X+\Var Y}{2}+\frac{A^2}{12}.
$$
\end{lem}
\nopagebreak[4]
\begin{proof}
Recall that we denote by $\Phi$ and $\Psi$ the CDFs corresponding to $\varphi$ and $\psi$, respectively. Integrating the CDFs gives
$$
 A=\int_0^1(\Phi-\Psi)
 =\E Y-\E X
$$
Since $H'=\varphi-\psi$ and $H(0)=H(1)=0$, integration by parts yields
$$ 3A\,\E(Z-m)^2=\int_0^1\bigl((x-m)^3\bigr)'H(x)\,dx=-\int_0^1(x-m)^3\bigl(\varphi(x)-\psi(x)\bigr)\,dx$$
$$=\E(Y-m)^3-\E(X-m)^3.$$
Now
$$
 X-m=(X-\E X)-\frac A2,
 \qquad
 Y-m=(Y-\E Y)+\frac A2.
$$
Expanding the two cubes and using
$\E(X-\E X)=\E(Y-\E Y)=0$, we obtain
$$
 3A\,\E(Z-m)^2=\E(Y-\E Y)^3-\E(X-\E X)^3+\frac{3A}{2}(\Var X+\Var Y)+\frac{A^3}{4}.
$$
The first two terms have nonpositive sum by
Lemma~\ref{lem:low_bound_enpdpoint}. Dividing by $3A$ proves the claim.
\end{proof}

\begin{proof}[Proof of Theorem~\ref{thm:endpoint}]
Let $X,Y,Z,m$ be as in Lemma~\ref{lem:upper_bound}. Let $P$ denote the CDF of $Z$. Since $Z$ has a density, $P(Z)$ is uniform
on $[0,1]$. Hence the Gini covariance identity \ref{eq:gini-cov} and $2uv\le u^2+v^2$ give
$$
\begin{aligned}
 \Gini(p)
 &=4\Cov(Z,P(Z))
 =4\E\bigl[(Z-m)(P(Z)-\tfrac12)\bigr]\\
 &\le2\E(Z-m)^2+2\E(P(Z)-\tfrac12)^2\\
 &=2\E(Z-m)^2+\frac16\le\Var X+\Var Y+\frac{A^2}{6}+\frac16,
\end{aligned}
$$
where the last inequality follows from Lemma~\ref{lem:upper_bound}.

On the other hand, since $q$ is an equal-probability mixture of laws of $X$ and $Y$,
equation \ref{eq:gini-mixture} from Lemma~\ref{lem:gini-identities} gives
\begin{equation}\label{eq:Dqidentity}
 \Gini(q)
 =\frac{\Gini(X)+\Gini(Y)}2
  +\frac12\int_0^1H(x)^2\,dx.
\end{equation}
 Now,
\eqref{eq:Dmunu} yields
\begin{equation}\label{eq:Dqlower}
 \Gini(q)\ge\Var X+\Var Y
 +\frac{\E X(1-\E X)+\E Y(1-\E Y)}3+\frac12\int_0^1H(x)^2\,dx.
\end{equation}
By Cauchy--Schwarz, $\int_0^1H^2\ge A^2$. In fact, the inequality is strict. Indeed, \(H\) is continuous, \(H(0)=H(1)=0\), and \(A>0\). Thus \(H\) cannot be constant almost everywhere. Subtracting the upper bound for $\Gini(p)$, and using $A=\E Y-\E X$, gives
$$
\begin{aligned}
 \Gini(q)-\Gini(p)
 &\ge\frac{\E X(1-\E X)+\E Y(1-\E Y)}{3}+\frac12\int_0^1H(x)^2\,dx-\frac{A^2}{6}-\frac16\\
 =&\frac{2\E X(1-\E X)+2\E Y(1-\E Y)+2A^2-1}{6}+\frac{1}{2}(\int_0^1H(x)^2\,dx-A^2)\\
 =&\frac{(1-2\E X)(2\E Y-1)}6 + \frac{1}{2}(\int_0^1H(x)^2\,dx-A^2) > 0.
\end{aligned}
$$
In the right-hand side of the final inequality the former summand is non-negative, as
$\E X=\E T/2\le1/2$ and $\E Y=1-\E S/2\ge1/2$, where $T$ and $S$ are as in Proposition~\ref{prop:monotone_factoriz} and Lemma~\ref{lem:low_bound_enpdpoint}.
\end{proof}

The analogous result for the second moment is immediate.
\begin{cor}\label{cor:sec_mom}
In the setting of Theorem~\ref{thm:endpoint}, let $Q$ be a random variable
with density $q$, and let $Z$, as before, have density $p$. Then
$\Var(Q)>\Var(Z)$.
\end{cor}
\begin{proof}
Since $q$ is an equal-probability mixture of $X$ and $Y$,
$$
 \E Q=m,
 \qquad
 \Var Q=\frac{\Var X+\Var Y}{2}+\frac{A^2}{4}.
$$
By Lemma~\ref{lem:upper_bound},
$$
 \Var Z\le\E(Z-m)^2
 \le\frac{\Var X+\Var Y}{2}+\frac{A^2}{12}
 <\Var Q.\qedhere
$$
\end{proof}

\section{Proof of Theorem~\ref{thm:main}}

Fix a direction $u$ that is not perpendicular to any
nondegenerate segment of $\partial K$, normalize its projection to $[0,1]$, and use the notation from the plan:
$$
 W=f-g,\qquad a=f',\qquad b=g'.
$$
The functions \(f,g\) are absolutely continuous and their derivatives are understood almost everywhere. The functions $a$ and $b$ are respectively nonincreasing and nondecreasing, while
$$
 W'=a-b,\qquad W(0)=W(1)=0.
$$
The normalized projected area density is
\begin{equation}\label{eq:area-density}
 p(x)=\frac{W(x)}{\mathcal A},
 \qquad
 \mathcal A=\int_0^1W(x)\,dx.
\end{equation}
The unnormalized projected arclength boundary density is
$$
 r(x)=\sqrt{1+a(x)^2}+\sqrt{1+b(x)^2}.
$$

Let
$$
R=\int_0^1r(x)\,dx,
$$
so that the normalized projected boundary density is
$
q(x)=\frac{r(x)}{R}.
$
Define
$$
\varphi(x)
=
\frac{
\sqrt{1+a(x)^2}+a(x)
+
\sqrt{1+b(x)^2}-b(x)
}{R}
$$
and
$$
\psi(x)
=
\frac{
\sqrt{1+a(x)^2}-a(x)
+
\sqrt{1+b(x)^2}+b(x)
}{R}.
$$
Clearly,
$$
q=\frac{\varphi+\psi}{2}.
$$
 Since $a$ is nonincreasing and $b$ is nondecreasing, it follows that
$\varphi$ is nonincreasing and $\psi$ is nondecreasing. Both are nonnegative. Moreover,
$$
\int_0^1(a-b)\,dx=W(1)-W(0)=0,
$$
and hence
$$
\int_0^1\varphi(x)\,dx
=
\frac{1}{R}\int_0^1\bigl(r(x)+a(x)-b(x)\bigr)\,dx
=1,
$$
and similarly
$$
\int_0^1\psi(x)\,dx=1.
$$
Thus $\varphi$ and $\psi$ are probability densities of the type appearing in
Theorem~\ref{thm:endpoint}.

Let $\Phi$ and $\Psi$ be their distribution functions. Then
$$
\Phi(x)-\Psi(x)
=
\frac{2}{R}\int_0^x(a(s)-b(s))\,ds
=
\frac{2W(x)}{R}.
$$
Consequently,
$$
\int_0^1\bigl(\Phi(x)-\Psi(x)\bigr)\,dx
=
\frac{2 \mathcal A}{R},
$$
and therefore
$$
\frac{\Phi(x)-\Psi(x)}
{\displaystyle\int_0^1(\Phi-\Psi)}
=
\frac{W(x)}{\mathcal A}
=
p(x).
$$
Since $q=(\varphi+\psi)/2$, Theorem~\ref{thm:endpoint} gives
$$
\Gini(q)> \Gini(p).
$$

There are at most countably many directions perpendicular
to nondegenerate segments of $\partial K$. For these
exceptional directions, the non-strict inequality
follows by continuity in $u$ from the inequality
already proved for all other directions. Indeed, for a fixed random vector $W$, the dominated convergence theorem implies that the function 
$$f_W(u) = \E |\langle W, u\rangle |$$
is continuous in $u$, and it remains to note that $D(q_u)$ and $D(p_u)$ are exactly of this form.
This completes the proof.

\begin{cor} \label{corol:conjecture_res}
For every planar convex body $K$,
$$
 \E |X_1-X_2| < \E |B_1-B_2|,
$$
where $X_1,X_2$ are uniform in $K$ and $B_1,B_2$ are uniform on $\partial K$. 
\end{cor}

\begin{proof}
Indeed, for normalized angular measure $\bar\sigma$ on $S^1$,
\begin{equation}\label{eq:cosine}
 |z|=\frac\pi2\int_{S^1}|\langle z,u\rangle|\,d\bar\sigma(u).
\end{equation}
See e.g. \cite{Schneider}. Integrating \eqref{eq:cosine} against the two-point area and boundary distributions, and then using \eqref{eq:directional}, gives the result. The inequality is strict as \eqref{eq:directional} is strict for almost every $u$.
\end{proof}

\begin{cor}\label{cor:sec_mom_planar}
For every planar convex body $K$,
$$
\mathbb E|X_1-X_2|^2 < \mathbb E|B_1-B_2|^2,
$$
where $X_1,X_2$ are uniform in $K$ and $B_1,B_2$ are uniform on $\partial K$.
\end{cor}

\begin{proof}
The proof of Theorem~\ref{thm:main}
shows that, for almost every $u\in S^1$, the projected area and boundary densities are
obtained from the same pair of nonincreasing and nondecreasing densities as in
Theorem~\ref{thm:endpoint}. Corollary~\ref{cor:sec_mom}
gives
$$
\operatorname{Var}\langle X_1,u\rangle
<
\operatorname{Var}\langle B_1,u\rangle.
$$
Observe that
$$
\mathbb E\langle X_1-X_2,u\rangle^2
=
2\operatorname{Var}\langle X_1,u\rangle,
$$
and similarly
$$
\mathbb E\langle B_1-B_2,u\rangle^2
=
2\operatorname{Var}\langle B_1,u\rangle.
$$
Therefore
$$
\mathbb E\langle X_1-X_2,u\rangle^2
<
\mathbb E\langle B_1-B_2,u\rangle^2
$$
for almost every $u\in S^1$.

For normalized angular measure $\bar\sigma$ on $S^1$,
$$
|z|^2
=
2\int_{S^1}\langle z,u\rangle^2\,d\bar\sigma(u).
$$
Integrating the preceding directional inequality over $u$ therefore gives
$$
\mathbb E|X_1-X_2|^2
<
\mathbb E|B_1-B_2|^2.
$$

\end{proof}

\section{A long-body observation}
In this section, we construct a three-dimensional counterexample  to the Zaporozhets-Tarasov conjecture, see Example~\ref{ex:counter-ex}. This example is then used to construct counterexamples in higher dimensions, see Proposition~\ref{prop:counter-ex}. The
main observation, discussed in Remark~\ref{rmk:why}, is that in dimension
three the area and perimeter profiles of sections orthogonal to a fixed
direction can be controlled almost independently. We exploit this
freedom to make the projected Gini mean difference strictly larger for
the volume measure than for the surface measure. A sufficiently large
stretch in the projection direction then yields the required
counterexample. 

\begin{lem}\label{lem:axis}
Suppose a random point has the form $Z=(LT,U)\in\R\times\R^m$, where $T\in[0,1]$ and the support of $U$ has diameter at most $D$. If $Z_1,Z_2$ are independent copies, then
$$
L\E|T_1-T_2|\leq \E|Z_1-Z_2|\leq L\E|T_1-T_2|+D.
$$
\end{lem}

\begin{proof}
The lower bound follows by projection onto the first coordinate. For the upper bound, use
$$
|(L(T_1-T_2),U_1-U_2)|\leq L|T_1-T_2|+|U_1-U_2|.
$$
\end{proof}

Thus, after division by $L$, mean distances in a long body converge to mean absolute differences of the corresponding axial coordinates.

\begin{ex}\label{ex:counter-ex}
    For $L>0$ and $\varepsilon>0$, define
$$
K_{L,\varepsilon}:=\left\{(Lt,x,y):0\leq t\leq1,\quad |x|\leq t,\quad |y|\leq\varepsilon\left(1-\frac{2t}{3}\right)\right\}.
$$
This is the convex hull of $(0,0,\pm\varepsilon)$ and the four points $(L,\pm1,\pm\varepsilon/3)$.
\end{ex}

\begin{figure}[t]
\centering
\includegraphics[width=0.55\textwidth]{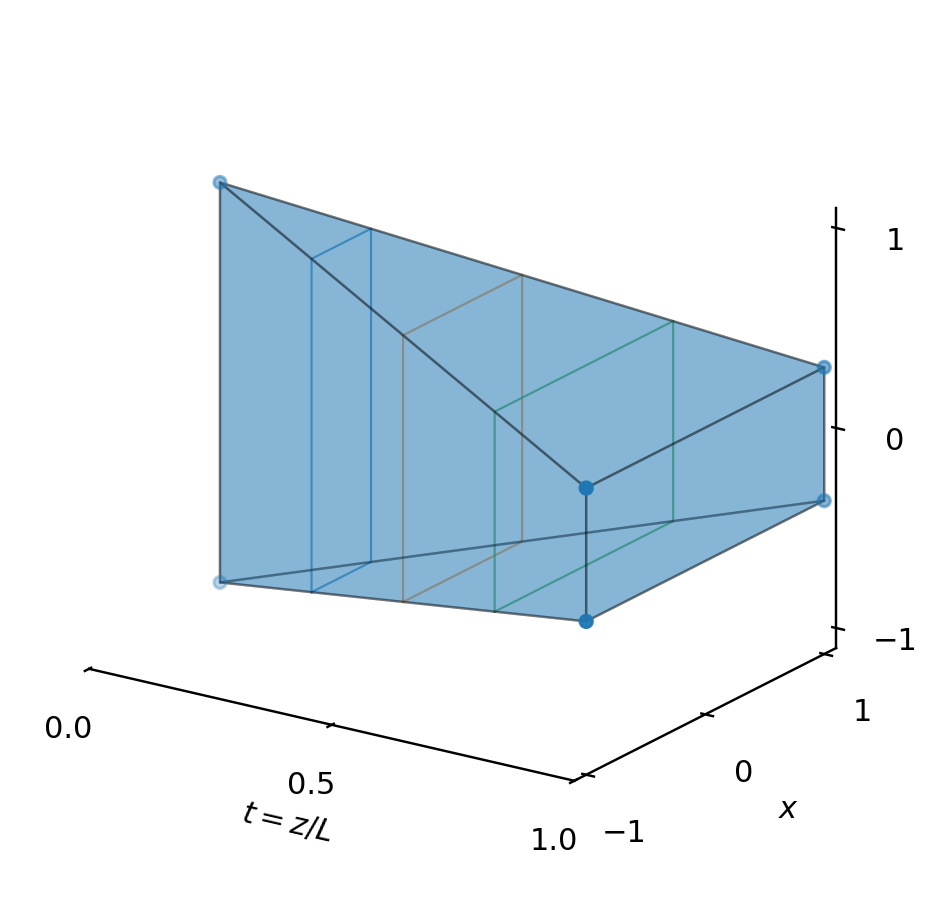}
\caption{The polytope $K_{L,\varepsilon}$ in the rescaled coordinates $(t,x,y/\varepsilon)$. One transverse width grows with $t$, while the other decreases.}
\label{fig:body}
\end{figure}

\begin{thm}\label{thm:explicit}
For $L=1000$ and $\varepsilon=1/100$, one has
$$
\Delta(K_{L,\varepsilon})>\theta(K_{L,\varepsilon}).
$$
\end{thm}

\begin{proof}

The area of the section at level $t$ equals
$
4\varepsilon t\left(1-\frac{2t}{3}\right).
$
Hence the axial coordinate of a uniform interior point has density
$$
p_I(t)=\frac{18}{5}t\left(1-\frac{2t}{3}\right)
$$
and distribution function
$$
F_I(t)=\frac{9t^2-4t^3}{5}.
$$
Using $\Gini(p)=2\int_0^1F(t)(1-F(t))\,dt$, we obtain
$$
\Gini(p_I)=\frac{246}{875}=0.281142857\ldots.
$$
Lemma~\ref{lem:axis} therefore gives
$$
\frac{\Delta(K_{L,\varepsilon})}{L}\geq\frac{246}{875}.
$$

We next compute the axial distribution of normalized surface measure. The two faces $x=\pm t$ have combined area density
$$
4\varepsilon\left(1-\frac{2t}{3}\right)\sqrt{L^2+1},
$$
while the two faces $y=\pm\varepsilon(1-2t/3)$ have combined area density
$$
4t\sqrt{L^2+\frac{4\varepsilon^2}{9}}.
$$
The end face $t=1$ has area $4\varepsilon/3$. Put
$$
a=\varepsilon\sqrt{L^2+1},\qquad
b=\sqrt{L^2+\frac{4\varepsilon^2}{9}}-\frac{2\varepsilon}{3}\sqrt{L^2+1},
$$
$$
D=a+\frac b2+\frac{\varepsilon}{3},\qquad
\alpha=\frac aD,\qquad \beta=\frac{b}{2D}.
$$
For $0\leq t<1$, the boundary axial distribution function is
$$
F_{\partial}(t)=\alpha t+\beta t^2;
$$
the remaining mass is an atom at $t=1$. Consequently its Gini mean difference is exactly
$$
g_{L,\varepsilon}
=\alpha+\frac{2\beta}{3}-\frac{2\alpha^2}{3}-\alpha\beta-\frac{2\beta^2}{5}.
$$
The diameter of the projection of $K_{L,\varepsilon}$ onto $x,y$ coordinates is at most $2\sqrt{1+\varepsilon^2}$, so Lemma~\ref{lem:axis} gives
$$
\frac{\theta(K_{L,\varepsilon})}{L}
\leq g_{L,\varepsilon}+\frac{2\sqrt{1+\varepsilon^2}}{L}.
$$
For the values suggested in the statement,
$
g_{1000,1/100}=0.2692731635\ldots
$
and hence
$$
\frac{\theta(K_{1000,1/100})}{1000}
\leq0.2712732636\ldots
<0.2811428571\ldots
\leq\frac{\Delta(K_{1000,1/100})}{1000}.
$$
\end{proof}

\begin{rem}\label{rmk:why}

Consider a family of convex bodies
\[
K_L=\left\{(Lt,z_1,\ldots,z_{d-1}):
0\le t\le1,\quad |z_i|\le a_i(t)\right\},
\]
where the functions $a_i$ are nonnegative, concave and
Lipschitz, and positive on $(0,1)$. The pushforward of
volume measure under $(Lt,z)\mapsto t$ has density
\[
2^{d-1}L\prod_{i=1}^{d-1}a_i(t).
\]
On the other hand, the pushforward of surface measure
restricted to the part of $\partial K_L$ with $0<t<1$
has density
\[
2^{d-1}\sum_{i=1}^{d-1}
\sqrt{L^2+a_i'(t)^2}
\prod_{j\ne i}a_j(t).
\]
The two end sections contribute only $O(1)$ to the
surface area. Consequently, as $L\to\infty$, the
normalized volume and surface measures, projected
onto the $t$-axis, converge to distributions with
densities proportional to
\[
\prod_{i=1}^{d-1}a_i(t)
\qquad\text{and}\qquad
\sum_{i=1}^{d-1}\prod_{j\ne i}a_j(t),
\]
respectively.

In dimension three this is the difference between a product $a_1a_2$ and a sum $a_1+a_2$. Choosing one width nondecreasing and one nonincreasing can make the normalized surface density more concentrated along the long axis than the normalized volume density. In particular, choosing
$a_1(t)=t$ and $a_2(t)=\varepsilon(1-2t/3)$
gives precisely the contrasting axial distributions
used in Example~\ref{ex:counter-ex}.  In dimension two the second expression is the empty product. 
\end{rem}

\begin{prop}\label{prop:counter-ex}
There are counterexamples in every dimension $d\geq3$.
\end{prop}
\begin{proof}
The case $d=3$ was established above. Put $K_L:=K_{L,1/100}$.
We first show that $\Delta(K_L)-\theta(K_L)$ grows linearly with $L$.
By Lemma~\ref{lem:axis} and the computations in the proof of
Theorem~\ref{thm:explicit},
\[
\lim_{L\to\infty}\frac{\Delta(K_L)}{L}=\frac{246}{875},
\]
whereas
\[
\lim_{L\to\infty}\frac{\theta(K_L)}{L}
=\left.\frac{4(19\varepsilon^2+33\varepsilon+9)}
{15(4\varepsilon+3)^2}\right|_{\varepsilon=1/100}
=\frac{93319}{346560}.
\]
Consequently,
\[
\lim_{L\to\infty}\frac{\Delta(K_L)-\theta(K_L)}{L}
=\frac{719927}{60648000}>0.
\]
Thus there exists $c>0$ such that
\[
\Delta(K_L)-\theta(K_L)\ge cL
\]
for all sufficiently large $L$.

Now let $d=3+m$, where $m\ge 1$, and put
\[
C_M:=[-M,M]^m,
\qquad
\widetilde K_{L,M}:=K_L\times C_M.
\]

The boundary of $\widetilde K_{L,M}$ is the union
$$
(\partial K_L\times C_M)\cup(K_L\times\partial C_M).
$$
The two sets intersect along \(\partial K_L\times\partial C_M\), which has zero \((d-1)\)-dimensional Hausdorff measure. Hence their surface measures add. The first part has surface measure
$$
\mathcal H^2(\partial K_L)(2M)^m,
$$
whereas the second has surface measure
$$
|K_L|\,\mathcal H^{m-1}(\partial C_M)
=
2m\,|K_L|(2M)^{m-1}.
$$
Here by $|K_L|$ we denote the volume of $K_L$. Elementary computations show \[
|K_L|=\frac L{90},
\qquad
\mathcal H^2(\partial K_L)\ge2L.
\]
Observe that the surface
area of $K_L\times\partial C_M$ normalized on the total surface area of $\widetilde K_{L,M}$ is \[
\eta_{L,M}=\frac{2m|K_L|(2M)^{m-1}}
{\mathcal H^2(\partial K_L)(2M)^m
+2m|K_L|(2M)^{m-1}}=\frac{m|K_L|}
{M\mathcal H^2(\partial K_L)+m|K_L|}\le\frac{m}{180M}.
\]

Let $\mu_{L,M}$ be normalized surface measure on
$\partial K_L\times C_M$. It is the product of normalized surface measure on
$\partial K_L$ and normalized volume measure on $C_M$. If $Z_1,Z_2$ are independent with law $\mu_{L,M}$, then, similarly to the proof of Lemma~\ref{lem:axis}, by the triangle inequality
$$
\mathbb E|Z_1-Z_2|
\le \theta(K_L)+2M\sqrt m.
$$

For $L\ge 2$, the diameter of $\widetilde K_{L,M}$ is
\[
D_{L,M}=\sqrt{L^2+1+\frac{1}{5625}+4mM^2}.
\]
Replacing normalized surface measure on
$\partial\widetilde K_{L,M}$ by $\mu_{L,M}$
changes the corresponding mean distance by at most
\[
2\eta_{L,M}D_{L,M}
\le \frac{m}{90M}
\sqrt{L^2+1+\frac{1}{5625}+4mM^2}
=O\left(\frac{L}{M}+1\right).
\]
Indeed, normalized surface measure on
$\partial\widetilde K_{L,M}$ is a mixture with weight
$1-\eta_{L,M}$ on $\mu_{L,M}$ and weight $\eta_{L,M}$
on the other boundary component. In the corresponding
two-point distribution, the contribution from
$\mu_{L,M}\otimes\mu_{L,M}$ has weight
$(1-\eta_{L,M})^2$. The remaining contribution has
weight $2\eta_{L,M}-\eta_{L,M}^2\le 2\eta_{L,M}$,
and all distances lie in $[0,D_{L,M}]$
Consequently,
$$
\theta(\widetilde K_{L,M})
\le
\theta(K_L)
+
O(M)
+
O\!\left(\frac{L}{M}+1\right).
$$
Choose $M=M(L)$ so that
$$
M\longrightarrow\infty,
\qquad
\frac{M}{L}\longrightarrow0
\qquad (L\to\infty),
$$
say, $M=\sqrt{L}$, then the error terms above are $o(L)$.
On the other hand, projection onto the first three coordinates does not increase distances, so
$$
\Delta(\widetilde K_{L,M})\ge \Delta(K_L).
$$

 Combining all inequalities above, we obtain
$$
\Delta(\widetilde K_{L,M})-\theta(\widetilde K_{L,M})
\ge
cL-o(L)>0
$$
for all sufficiently large $L$. Thus $\widetilde K_{L,M}$ is a counterexample in
dimension $d$.
\end{proof}

\section{Note added}
 After the first version of our work was posted on arXiv, two related independent works became publicly available. Kukushkin \cite{Kukushkin_mach} gave a proof of Theorem~\ref{thm:endpoint} and the planar Zaporozhets--Tarasov conjecture using a quite different approach. Lotnikov \cite{Lotnikov_thes}, whose work is based on his bachelor's thesis defended in June 2026, constructed centrally symmetric counterexamples in every dimension $d\ge3$, for every moment of the distance. These results overlap, respectively, with our Theorem~\ref{thm:main} and the higher-dimensional counterexamples.

 \section{Acknowledgements}
 The author thanks A. Lotnikov and D. Zaporozhets for useful discussions.

 \section{Declaration of AI use}
 The author used ChatGPT with the GPT-5.6 Sol model during the preparation of
this work. The first proof of the planar case of Zaporozhets-Tarasov conjecture, namely Theorem~\ref{thm:main}, was obtained
by the author without the use of AI tools; it was based on KKT conditions and an
exhaustive case analysis and was substantially more technical. The argument was
subsequently simplified through discussions with ChatGPT. The author also used
ChatGPT to search for the explicit construction in Example 4.2 from the author's
underlying idea described in Remark~\ref{rmk:why} (that includes generation of relevant
code). The author reviewed and verified all AI-assisted output and takes full
responsibility for the content of the article.

\end{document}